\documentclass[12pt]{article}
\usepackage{geometry}
\usepackage{amssymb}
\usepackage{amsmath}
\usepackage{float}
\usepackage{multirow}
\usepackage{color}
\usepackage{graphicx}
\usepackage{subfigure}
\usepackage{tabu}
\usepackage{booktabs}
\usepackage{array}
\usepackage{caption}
\usepackage{amsthm}
\usepackage{booktabs}
\usepackage{cite}
\usepackage{indentfirst}
\numberwithin{equation}{section}
\newtheorem{theorem}{\bf Theorem}[section]
\usepackage[noend]{algpseudocode}
\usepackage{algorithmicx,algorithm}
\newtheorem{lemma}{Lemma}[section]

\newtheorem{Definition}{Definition}[section]

\newcommand{\xqedhere}[2]{%
\rlap{\hbox to#1{\hfil\llap{\ensuremath{#2}}}}}
\makeatletter
\@addtoreset{equation}{section}
\makeatother

\title{Incremental tensor regularized least squares with multiple right-hand sides}
\author{
	Zhengbang Cao\footnote{School of Mathematical Sciences, Ocean University of China, Qingdao 266100, China.
		E-Mail: {\tt caozhengbang@stu.ouc.edu.cn}},
	Pengpeng Xie\footnote{Corresponding author: School of Mathematical Sciences, Ocean University of China, Qingdao 266100, China.
		E-Mail: {\tt xie@ouc.edu.cn}. 
	}
}
\date{}
\begin{document}
	\maketitle
\begin{abstract}
Solving linear discrete ill-posed problems for third order tensor equations based on a tensor t-product has attracted much attention. 
But when the data tensor is produced continuously, current algorithms are not time-saving.
	Here, we propose an incremental tensor regularized least squares (t-IRLS) algorithm  with the t-product that incrementally computes the solution to the tensor regularized least squares (t-RLS) problem with multiple lateral slices on the right-hand side. More specifically, we update its solution by solving a t-RLS problem with a single lateral slice on the right-hand side 
	whenever a new horizontal sample arrives, instead of solving the t-RLS problem from scratch.
The t-IRLS algorithm is well suited for large data sets and real time operation.
 Numerical examples are presented to demonstrate the efficiency of our algorithm.
	\\ \hspace*{\fill}  
    \\ \hspace*{\fill} 
\end{abstract}	
\section{Introduction}

\hskip 2em 
We consider the tensor regularized least squares (t-RLS) problem with multiple lateral slices on the right-hand side
\begin{equation}\label{tensor Tikhonov regularization}
	\mathop{\min}\limits_{\mathcal{X}\in\mathbb{R}^{n\times c\times p}}
	\left\{ {\left\|\mathcal{A}*{\mathcal{X}}-{\mathcal{B}}\right\|}^2_F + \lambda^2{\left\|{\mathcal{L}*\mathcal{X}}\right\|}^2_F\right\},
\end{equation}
where $\mathcal{A}\in\mathbb{R}^{m\times n\times p}$, $\mathcal{B}\in\mathbb{R}^{m\times c\times p}$, $\mathcal{X}\in\mathbb{R}^{n\times c\times p}$ with $c>1$ and $\lambda>0$ is a regularization parameter. Several forms of the regularization operator $\mathcal{L}$ are presented in \cite{Lothar2021,ReichelGolubKahanTikhonov2021} and here we focus on $\mathcal{L}=\mathcal{I}$, an identity tensor 
throughout this paper. The operator $*$ denotes the tensor-tensor t-product introduced in the seminal work \cite{Kilmer2011,Kilmer2013}
, which has been proved to be a useful tool with a large number of applications, such as image processing \cite{Kilmer2013,Martin2013AnOT,Soltani2015,Tarzanagh2018}, signal processing \cite{Chan2016,Liu2018ImprovedRT,Long2019LowRT}, tensor recovery and robust tensor PCA \cite{Kong2018tSchattenp,Liu2018ImprovedRT}, data completion and denoising \cite{Hu2017MovingOD,Long2019LowRT,Wang2019NoisyLT}. 
  Compared to the tensor least squares (t-LS) of the form
\begin{equation*}
	\mathop{\min}\limits_{\mathcal{X}\in\mathbb{R}^{n\times c\times p}} \left\{{\left\|\mathcal{A}*{\mathcal{X}}-{\mathcal{B}}\right\|}^2_F\right\},
\end{equation*}
 in which the severe ill-conditioning of $\mathcal{A}$ and the error in $\mathcal{B}$ may cause a large propagated error in computing the solution, solving a nearby problem (\ref{tensor Tikhonov regularization}) can give a much more meaningful approximation \cite{Lothar2021,ReichelGolubKahanTikhonov2021}.
 Several methods, such as tensor Golub-Kahan method  \cite{Reichel2021WeightedTG,2021ElGuide,ReichelGolubKahanTikhonov2021,ReichelApplNM2021}, tensor Arnoldi method \cite{Lothar2021,ReichelApplNM2021}, tensor GMRES algorithm \cite{2021ElGuide,ReichelApplNM2021} and tensor generalized singular value decomposition algorithm \cite{Zhang2021CSDA} are researched for the t-RLS problem.

 \hskip 2em
 In practical applications, we encounter some cases where only partial data sets are known at first, and the new data sets are available in the next time step or are arriving continuously over time. For this type of situation, incremental algorithms are particularly effective and  various methods have been studied. The general framework of incremental tensor analysis, which efficiently computes a compact summary for high-order and high-dimensional data was developed in \cite{Sun2008IncrementalTA}. Zhang et al. \cite{ZhangXIao2016} proposed the incremental algorithm that incrementally computes the solution to the regularized least squares problem. In addition, Zeng and Ng \cite{Zeng2021} designed the incremental CANDECOMP/PARAFAC tensor decomposition by an alternating minimization method. 
 In this paper, we are interested in the case that a new horizontal sample is acquired in the t-RLS problem (\ref{tensor Tikhonov regularization}).
 Obviously, in such a case, conventional algorithms that typically compute the result from scratch whenever a new horizontal sample comes are highly inefficient.
 Inspired by the novel algorithm presented in \cite{ZhangXIao2016}, we introduce an incremental algorithm t-IRLS that incrementally computes the solution to the t-RLS problem.
 In essence, to compute the new solution when a new sample arrives,
 our algorithm updates the old one by solving a t-RLS problem with a single lateral slice on the right-hand side,
  instead of solving the t-RLS problem with $c$ lateral slices on the right-hand side.
  Therefore, the computational cost of our algorithm becomes more pronounced when $c$ is large.

 \hskip 2em The paper is organized as follows. In Section 2, we review basic definitions and notations. Section 3 recalls the tensor Golub-Kahan-Tikhonov (t-GKT) algorithm and details the t-IRLS algorithm to solve the t-RLS problem. We perform numerical experiments to check the effectiveness and efficiency of the t-IRLS algorithm 
 in Section 4. Finally, we draw some conclusions in the last section.
\section{Preliminaries}	
\hskip 2em We start this section with fundamental notions and properties of tensors based on the t-product.
\subsection{Notation and indexing}
\hskip 2em Throughout this paper, real third-order tensors denoted by calligraphic script letters are considered. Capital letters refer to matrices, and lower bold case letters to vectors. The $i$th frontal slice and $j$th lateral slice of tensor $\mathcal{A}$ will be denoted by $\mathcal{A}^{(i)}$ and $\vec{\mathcal{A}}_j$ respectively. 
An element $\mathbf{a}\in\mathbb{R}^{1\times 1\times n}$ is called a tubal scalar of length $n$, which will play a role similar to scalars in $\mathbb{R}$ and $\mathbf{0}$ refers to the tensor whose frontal slices are all zeros.

\hskip 2em
For $\mathcal{A}\in \mathbb{R}^{n_1\times n_2\times n_3}$ with $n_3$ frontal slices $\mathcal{A}^{(i)}$, then
\begin{equation*}
	\texttt{bcirc}(\mathcal{A})=\begin{bmatrix}
		\mathcal{A}^{(1)}&\mathcal{A}^{(n_3)}&\cdots&\mathcal{A}^{(2)}\\
		\mathcal{A}^{(2)}&\mathcal{A}^{(1)}&\cdots&\mathcal{A}^{(3)}\\
		\vdots&\vdots&\ddots&\vdots\\
		\mathcal{A}^{(n_3)}&\mathcal{A}^{(n_3-1)}&\cdots&\mathcal{A}^{(1)}
	\end{bmatrix}
\end{equation*}
is a block  circulant matrix of size $n_1n_3\times n_2n_3$.
The command \texttt{unfold} reshapes a tensor $\mathcal{A}\in \mathbb{R}^{n_1\times n_2 \times n_3}$ into an $n_1n_3\times n_2$ matrix, whereas the $\mathtt{fold}$ command is the inverse defined as
\begin{equation*}
	\mathtt{unfold}({\mathcal{A}})=\begin{bmatrix}
		\mathcal{A}^{(1)}\\\mathcal{A}^{(2)}\\ \vdots \\\mathcal{A}^{(n_3)}
	\end{bmatrix},\ \mathtt{fold}(\mathtt{unfold}(\mathcal{A}))=\mathcal{A}.
\end{equation*}
\subsection{Discrete Fourier transform}	
\hskip 2em The discrete Fourier transform (DFT) on $v\in \mathbb{R}^n$, denoted as $\bar{v}$, is given by
$\bar{v}=F_nv\in \mathbb{C}^n$. 
Here $F_n$ is the DFT matrix
\begin{equation*}
	F_n=
	\begin{bmatrix}
		1&1&1&\cdots&1\\
		1&\omega&\omega^2&\cdots&\omega^{n-1}\\
		\vdots&\vdots&\vdots&\ddots&\vdots\\
		1&\omega^{n-1}&\omega^{2(n-1)}&\cdots&\omega^{(n-1)(n-1)}
	\end{bmatrix},
\end{equation*}
where $\omega=e^{\frac{-2\pi \mathtt i}{n}}$ is a primitive $n$th root of unity with  $\texttt{i}=\sqrt{-1}$ and $F_n$ satisfies $F_n^{\mathrm{H}}F_n=F_nF_n^{\mathrm{H}}=nI_n.$ The block circulant matrix can be block diagonalized by the DFT, i.e.,
\begin{equation}\label{operator}
	(F_{n_3}\otimes I_{n_1})\cdot\texttt{bcirc}(\mathcal{A})\cdot(F_{n_3}^{-1}\otimes I_{n_2})=\bar{A},
\end{equation}
where $\otimes$ denotes the Kronecker product and 	$\bar{A}=\texttt{diag}(\bar{\mathcal{A}}^{(1)},\bar{\mathcal{A}}^{(2)},\ldots,\bar{\mathcal{A}}^{(n_3)})$.

\hskip 2em
Actually, taking the fast Fourier transform (FFT) along each tubal scalar of $\mathcal{A}$ generates a new tensor $\bar{\mathcal{A}}$ with frontal slices $\bar{\mathcal{A}}^{(i)}$,
\begin{equation*}
	\bar{\mathcal{A}}=\mathtt{fft}(\mathcal{A},[],3)=\mathtt{fold}\left(\begin{bmatrix}
		\bar{\mathcal{A}}^{(1)}\\\bar{\mathcal{A}}^{(2)}\\\vdots\\\bar{\mathcal{A}}^{(n_3)}
	\end{bmatrix}\right).
\end{equation*}


\subsection{Definitions and propositions}
\begin{Definition}{(t-product) \cite{Kilmer2013}}
	Let $\mathcal{A}\in \mathbb{R}^{n_1\times n_2 \times n_3}$ and $\mathcal{B}\in \mathbb{R}^{n_2\times n_4 \times n_3}$. The t-product $\mathcal{A}*\mathcal{B}$ is the tensor $\mathcal{C}\in \mathbb{R}^{n_1\times n_4 \times n_3}$ defined by
	\begin{equation*}
		\mathcal{C}=\mathtt{fold}(\mathtt{bcirc}(\mathcal{A})\cdot \mathtt{unfold}(\mathcal{B})).
	\end{equation*}
\end{Definition}
Notice that the t-product reduces to the standard matrix multiplication when $n_3=1$. The following lemma is important for conducting t-product based computations.
\begin{lemma}\label{relation}\cite{Kilmer2013}
	For $\mathcal{A}$, $\mathcal{B}$ and $\mathcal{C}$ of appropriate size, the following statements hold:
	\begin{equation*}
		\mathcal{C}=\mathcal{A}*\mathcal{B} \iff \bar{C}=\bar{A}\cdot\bar{B},\quad
		\mathcal{C}=\mathcal{A}+\mathcal{B} \iff \bar{C}=\bar{A}+\bar{B}.
	\end{equation*}
\end{lemma}	
\begin{Definition}{(identity tensor) \cite{Kilmer2013}}
	The identity tensor $\mathcal{I}\in \mathbb{R}^{n\times n\times n_3}$ is the tensor with
	$I^{(1)}$ being the $n\times n$ identity matrix, and other frontal slices being zeros.
\end{Definition}
\begin{Definition}{(tensor transpose) \cite{Kilmer2013}}
	If $\mathcal{A}\in \mathbb{R}^{n_1\times n_2\times n_3}$, then $\mathcal{A}^{\mathrm{T}}$ is the $n_2\times n_1\times n_3$ tensor obtained by transposing each of the frontal slices and then reversing the order of transposed frontal slices 2 through $n_3$.	
\end{Definition}
\begin{Definition}{(inverse tensor) \cite{Kilmer2013}}\label{t-inverse}
	An $n\times n\times n_3$ tensor $\mathcal{A}$ has an inverse $\mathcal{B}$, provided that $\mathcal{A}*\mathcal{B}=\mathcal{I}_{nnn_3}$ and $ \mathcal{B}*\mathcal{A}=\mathcal{I}_{nnn_3}.$
\end{Definition}
\begin{Definition}(orthogonal tensor)\cite{Kilmer2013}
An $n\times n \times n_3$ real-valued tensor $\mathcal{Q}$ is orthogonal if $\mathcal{Q}^{\mathrm{T}}*\mathcal{Q}=\mathcal{Q}*\mathcal{Q}^{\mathrm{T}}=\mathcal{I}_n$.
\end{Definition}
\begin{Definition}\cite{Kilmer2013}
	For $\mathcal{A}\in \mathbb{R}^{n_1\times n_2\times n_3}$,
	the Frobenius norm of $\mathcal{A}$ is defined as
	\begin{equation*}
		{\left\|\mathcal{A}\right\|}_F=\sqrt{\sum_{ijk}{\left|a_{ijk}\right|}^2}.
	\end{equation*}
The length of any nonzero $\vec{\mathcal{X}}\in \mathbb{R}^{n\times 1\times n_3}$ is given as
\begin{equation*}
	\|\vec{\mathcal{X}}\|=\frac{\|\vec{\mathcal{X}}^{\mathrm{T}}*\vec{\mathcal{X}}\|_F}{\|\vec{\mathcal{X}}\|_F}.
\end{equation*}
\end{Definition}
\begin{Definition}{(Moore-Penrose inverse of tensor) \cite{Jin2017}}
Let $\mathcal{A}\in \mathbb{R}^{n_1\times n_2 \times n_3}$. If there exists a tensor $\mathcal{X}\in \mathbb{R}^{n_2\times n_1\times n_3 }$ such that
\begin{equation*}		\mathcal{A}*\mathcal{X}*\mathcal{A}=\mathcal{A},\ \mathcal{X}*\mathcal{A}*\mathcal{X}=\mathcal{X},\ {(\mathcal{A}*\mathcal{X})}^{\mathrm{T}}=\mathcal{A}*\mathcal{X},\ {(\mathcal{X}*\mathcal{A})}^{\mathrm{T}}=\mathcal{X}*\mathcal{A},
\end{equation*}
then $\mathcal{X}$ is called the Moore-Penrose inverse of the tensor $\mathcal{A}$ and is denoted by $\mathcal{A}^{\dagger}$.
\end{Definition}
\begin{theorem}(t-SVD and t-QR)\cite{Martin2013AnOT}
Let $\mathcal{A}$ be an $n_1\times n_2\times n_3$ real-valued tensor. Then $\mathcal{A}$ can be factored as
\begin{equation*}
	\mathcal{A}=\mathcal{U}*\mathcal{S}*\mathcal{V}^{\mathrm{T}},\ \mathrm{and}\  \mathcal{A}=\mathcal{Q}*\mathcal{R},
\end{equation*}
where $\mathcal{U}\in\mathbb{R}^{n_1\times n_1\times n_3}$, $\mathcal{V}\in\mathbb{R}^{n_2\times n_2\times n_3}$, $\mathcal{Q}\in\mathbb{R}^{n_1\times n_1\times n_3}$ are orthogonal tensor, $\mathcal{S}\in\mathbb{R}^{n_1\times n_2\times n_3}$ is f-diagonal, i.e., each of its frontal slice is a diagonal matrix, and $\mathcal{R}\in\mathbb{R}^{n_1\times n_2\times n_3}$ is f-upper triangular, whose frontal slices are all upper triangular matrix. Specially, the tensor on the diagonal of $\mathcal{S}$ is called the singular tube.

\end{theorem}

\hskip 2em
Analogous to the normalization operations in the matrix cases, Algorithm 1 takes a nonzero tensor $\vec{\mathcal{X}}\in\mathbb{R}^{n\times 1\times n_3}$ and returns a normalized tensor $\vec{\mathcal{V}}\in\mathbb{R}^{n\times 1\times n_3}$ and a tubal scalar $\mathbf{a}\in\mathbb{R}^{1\times 1\times n_3}$ such that
\begin{equation*}
	\vec{\mathcal{X}}=\vec{\mathcal{V}}*\mathbf{a}\ \mathrm{and}\ \|\vec{\mathcal{V}}\|=1.
\end{equation*}
\begin{algorithm}[htb]
	\caption{Normalize \cite{Kilmer2013}}
	\hspace*{0.02in} {\bf Input:}
	$\vec{\mathcal{X}}\in {\mathbb{R}}^{ n\times 1\times n_3}\neq \mathbf{0}$\\
	\hspace*{0.02in} {\bf Output:}
    $\vec{\mathcal{V}},\mathbf{a}$ with $\vec{\mathcal{X}}=\vec{\mathcal{V}}*\mathbf{a}$ and $\|\vec{\mathcal{V}}\|=1$
	\begin{algorithmic}[1]
		\State $\vec{\mathcal{V}}\ \gets\ \mathtt{fft}(\vec{\mathcal{X}},[],3)$
		\For{$j=1,2,\ldots,n_3$}
		\State $\mathbf{a}^{(j)}$ $\gets$ $\|\vec{\mathcal{V}}^{(j)}\|_2$
		\If{$\mathbf{a}^{(j)}\ge\ \mathtt{tol}$}
		\State $\vec{\mathcal{V}}^{(j)}$ $\gets$ $\frac{1}{\mathbf{a}^{(j)}}\vec{\mathcal{V}}^{(j)}$
		\Else
		\State $\vec{\mathcal{V}}^{(j)}$ $\gets$ $\mathtt{randn}(n,1)$;
		$\mathbf{a}^{(j)}$ $\gets$ $\|\vec{\mathcal{V}}^{(j)}\|_2$;
		$\vec{\mathcal{V}}^{(j)}$ $\gets$ $\frac{1}{\mathbf{a}^{(j)}}\vec{\mathcal{V}}^{(j)}$;
		$\mathbf{a}^{(j)}$ $\gets$ $\mathbf{0}$
		\EndIf
		\EndFor
		\State $\vec{\mathcal{V}}$ $\gets$ $\mathtt{ifft}(\vec{\mathcal{V}},[],3)$;
		$\mathbf{a}^{(j)}$ $\gets$ $\mathtt{ifft}(\mathbf{a},[],3)$
	\end{algorithmic}
\end{algorithm}
\begin{algorithm}[htb]
	\caption{tensor Golub-Kahan bidiagonalization \cite{Kilmer2013}}
	\hspace*{0.02in} {\bf Input:}
	$\mathcal{A}\in {\mathbb{R}}^{ n_1\times n_2\times n_3}$,
	$\vec{\mathcal{B}}\in {\mathbb{R}}^{ n_1\times 1\times n_3}$ such that $\mathcal{A}^{\mathrm{T}}*\vec{\mathcal{B}}\neq\mathbf{0}$
	\begin{algorithmic}[1]
		\State $\vec{\mathcal{W}}_0$ $\gets$ $\mathbf{0}$
		\State $[\vec{\mathcal{Q}}_1,\mathbf{z_1}]\gets\mathtt{Normalize}(\vec{\mathcal{B}})$ with $\mathbf{z_1}$ invertible
		\For{$i=1,2,\cdots,k$}
		\State
		$\begin{cases}
			\vec{\mathcal{W}}_i \gets \mathcal{A}^{\mathrm{T}}*\vec{\mathcal{Q}_i}-\vec{\mathcal{W}}_{i-1}*\mathbf{z}_i\ (\mathrm{no\ reorthogonalization})\\
			\\
			\vec{\mathcal{W}}_i$ $\gets$ $\mathcal{A}^{\mathrm{T}}*\vec{\mathcal{Q}_i}-\vec{\mathcal{W}}_{i-1}*\mathbf{z}_i,
			\vec{\mathcal{W}}_i \gets
			\vec{\mathcal{W}}_i-\sum\limits_{j=1}^{i-1}\vec{\mathcal{W}}_j*(\vec{\mathcal{W}}_j^{\mathrm{T}}*\vec{\mathcal{W}}_i)\\ (\mathrm{with\ reorthogoonalization}) 				
		\end{cases}$
		\State $[\vec{\mathcal{W}}_i,\mathbf{c}_i]$ $\gets$ $\mathtt{Normalize}(\vec{\mathcal{W}}_i)$
		\State
		$\begin{cases}
			\vec{\mathcal{Q}}_{i+1}\gets \mathcal{A}*\vec{\mathcal{W}}_i-\vec{\mathcal{Q}}_i*\mathbf{c}_i\ (\mathrm{no  \ reorthogonalization})\\
			\\
			\vec{\mathcal{Q}}_{i+1} \gets \mathcal{A}*\vec{\mathcal{W}}_i-\vec{\mathcal{Q}}_i*\mathbf{c}_i,
			\vec{\mathcal{Q}}_{i+1}\gets \vec{\mathcal{Q}}_{i+1}-\sum\limits_{j=1}^{i}\vec{\mathcal{Q}}_j*(\vec{\mathcal{Q}}_j^{\mathrm{T}}*\vec{\mathcal{Q}}_{i+1})\\ (\mathrm{with\ reorthgonalization})
		\end{cases}$
		\State $[\vec{\mathcal{Q}}_{i+1},\mathbf{z}_{i+1}]$ $\gets$ $\mathtt{Normalize}(\vec{\mathcal{Q}}_{i+1})$
		\EndFor
	\end{algorithmic}
\end{algorithm}

\hskip 2em
The tensor Golub-Kahan bidiagonalization (t-GKB) algorithm (Algorithm 2), introduced in \cite{Kilmer2013}, produces the t-GKB decomposition
\begin{equation*}
	\mathcal{A}*\mathcal{W}_k=\mathcal{Q}_{k+1}*\bar{\mathcal{P}}_k,\ \mathcal{A}^{\mathrm{T}}*\mathcal{Q}_k=\mathcal{W}_k*\mathcal{P}_k^{\mathrm{T}},
\end{equation*}
where
\begin{equation}
	\bar{\mathcal{P}}_{k}=\left[\begin{array}{ccccc}
		\mathbf{c}_{1} & & & & \\
		\mathbf{z}_{2} & \mathbf{c}_{2} & & & \\
		& \mathbf{z}_{3} & \mathbf{c}_{3} & & \\
		& & \ddots & \ddots & \\
		& & & \mathbf{z}_{k} & \mathbf{c}_{k} \\
		& & & & \mathbf{z}_{k+1}
	\end{array}\right] \in \mathbb{R}^{(k+1) \times k \times n_3},
\end{equation}
and tensor $\mathcal{P}_k$ represents the leading $k\times k \times n_3$ subtensor of $\bar{\mathcal{P}}_k$.

\hskip 2em
\section{Incremental tensor regularized least squares }
\hskip 2em  
In this section, we first recall 
the t-GKT algorithm for solving the t-RLS problem. With the addition of a new horizontal sample, the cost of  conventional algorithms increases accordingly. Thus attention is paid to establish an efficient algorithm that incrementally computes the solution to the t-RLS problem with multiple lateral slices on the right-hand side.
\subsection{Tensor regularized least squares}

\hskip 2em It is shown in \cite{Kilmer2013, Lothar2021,ReichelGolubKahanTikhonov2021} that the t-RLS problem (\ref{tensor Tikhonov regularization}) has a unique solution ${\mathcal{X}^*}$, 
\begin{equation}\label{normal equation c=1}
	{\mathcal{X}^*}=\left(\mathcal{A}^{\mathrm{T}}*\mathcal{A}+\lambda^2\mathcal{I}_n\right)^{-1}*\mathcal{A}^{\mathrm{T}}*{\mathcal{B}}
	=\mathcal{A}^{\mathrm{T}}*\left(\mathcal{A}*\mathcal{A}^{\mathrm{T}}+\lambda^2\mathcal{I}_m\right)^{-1}*\mathcal{B}.
\end{equation}
When it comes to designing a high-performance tensor algorithm, it is not enough simply to compute ${\mathcal{X}^*}$ by (\ref{normal equation c=1}). 

\hskip 2em
  Recalling the t-GKB algorithm, if the number of steps $k$ of the process is small enough to avoid breakdown, that is, $k$ is chosen small enough so that the tubal scalars $\mathbf{c}_i$ and $\mathbf{z}_{i+1}$ for $i=1,2,\cdots,k$, determined by Algorithm
  2, are invertible, then $\mathcal{A}$ can be reduced to a small bidiagonal tensor. Thus, by utilizing the t-GKB algorithm, the t-GKT method (Algorithm 4), discussed in \cite{ReichelGolubKahanTikhonov2021}, simplifies (\ref{tensor Tikhonov regularization}) to a t-LS problem with lower dimensions, which can be easily computed by Algorithm 3.
  \begin{algorithm}[htb]
  	\caption{Solution of a generic tensor least squares problem \cite{ReichelGolubKahanTikhonov2021}}
  	\hspace*{0.02in} {\bf Input:}
  	$\mathcal{C}\in {\mathbb{R}}^{ n_1\times n_2\times n_3}$, where its Fourier transform has nonsingular frontal slices;
  	$\vec{\mathcal{D}}\in {\mathbb{R}}^{ n_1\times 1\times n_3}$, $\vec{\mathcal{D}}\neq \mathbf{0}$
  	\\
  	\hspace*{0.02in} {\bf Output:}
  	The solution $\vec{\mathcal{Y}}\in {\mathbb{R}}^{ n_2\times 1\times n_3}$ of $\mathrm{min}_{\vec{\mathcal{Y}}\in {\mathbb{R}}^{ n_2\times 1\times n_3}}\|\mathcal{C}*\vec{\mathcal{Y}}-\vec{\mathcal{D}}\|_F$
  	\begin{algorithmic}[1]
  		\State $\mathcal{C}\gets\mathtt{fft}(\mathcal{C},[],3)$
  		\State $\vec{\mathcal{D}}\gets\mathtt{fft}(\vec{\mathcal{D}},[],3)$
  		\For{$i=1,2,\cdots,n_3$}
  		\State $\vec{\mathcal{Y}}(:,:,i)=\mathcal{C}(:,:,i)\backslash\vec{\mathcal{D}}(:,:,i)$, where $\backslash$ denotes MATLAB's backslash operator
  		\EndFor
  		\State $\vec{\mathcal{Y}}\gets\mathtt{ifft}(\vec{\mathcal{Y}},[],3)$
  	\end{algorithmic}
  \end{algorithm}
\begin{algorithm}[htb]
	\caption{t-GKT method for solving the t-RLS problem}
	\hspace*{0.02in} {\bf Input:}
	$\mathcal{A}\in {\mathbb{R}}^{ m\times n\times p}$,
	${\mathcal{B}}\in {\mathbb{R}}^{ m\times c\times p}$, $k\ge 1$
	\\
	\hspace*{0.02in} {\bf Output:}
	The solution ${\mathcal{X}}\in {\mathbb{R}}^{ n\times c\times p}$ of the t-RLS problem 
\begin{algorithmic}[1]
	\For{$j=1,2,\cdots,p$}
	\State $[\vec{\mathcal{Q}}_1,\mathbf{z_1}]\gets \mathrm{Normalize}(\vec{\mathcal{B}}_j)$
	\State Compute $\mathcal{W}_k,\mathcal{Q}_{k+1}$ and $\bar{\mathcal{P}}_k$ by Algorithm 2
	\State Compute the t-QR component f-upper triangular tensor $\mathcal{R}_{k}$ of $\mathcal{W}_{k}$
	\State Compute $\tilde{\mathcal{P}}_k\gets\bar{\mathcal{P}}_k*\mathcal{\mathcal{R}}_k^{-1}$
	\State Compute the solution $\vec{\mathcal{Z}}_j$ of the t-LS problem
	\begin{equation*}
		\mathrm{min}\left\{\left\|\begin{bmatrix}
			\tilde{\mathcal{P}}_k\\\lambda\mathcal{I}
		\end{bmatrix}*\vec{\mathcal{Z}}-\begin{bmatrix}
		\vec{e}_1*\mathbf{z}_1\\ \mathbf{0}
	\end{bmatrix}\right\|_F:\ \vec{\mathcal{Z}}\in{\mathbb{R}}^{ k\times 1\times p}\right\}
	\end{equation*}
	by Algorithm 3, where $\vec{e}_1$ is such that the $(1,1,1)$th entry of $\vec{e}_1$ equals $1$ and the remaining entries vanish
	\State Compute $\vec{\mathcal{X}}_j\gets\mathcal{W}_k*\mathcal{R}_k^{-1}*\vec{\mathcal{Z}}_j$
	\EndFor
\end{algorithmic}
\end{algorithm}
\hskip 2em However, when a new sample $\mathcal{A}_1\in\mathbb{R}^{n\times 1\times p}$ and the corresponding response ${\mathcal{B}_1}\in\mathbb{R}^{c\times 1\times p}$ are added, a new problem arises as shown below:
\begin{equation}\label{Tensor regularization new}
	\mathop{\min}\limits_{\mathcal{X}\in\mathbb{R}^{n\times c\times p}}
	\left\{ {\left\|\tilde{\mathcal{A}}*{\mathcal{X}}-{\tilde{\mathcal{B}}}\right\|}^2_F + {\lambda^2}{\left\|{\mathcal{X}}\right\|}^2_F\right\},
\end{equation}
where $\mathcal{\tilde{A}}=\begin{bmatrix}\mathcal{A}\\{\mathcal{A}^{\mathrm{T}}_1}\end{bmatrix}\in\mathbb{R}^{(m+1)\times n\times p},\mathcal{\tilde{B}}=\begin{bmatrix}\mathcal{B}\\{\mathcal{B}^{\mathrm{T}}_1}\end{bmatrix}\in \mathbb{R}^{(m+1)\times c \times p}$.
In such a case, the t-GKT method and other conventional algorithms that compute the result from scratch without making use of the known solution of the original problem (\ref{tensor Tikhonov regularization}) is inefficient. To conquer this problem, an incremental algorithm for solving (\ref{Tensor regularization new}) with lower cost is derived in the subsequent work.

\subsection{Our contribution}

\hskip 2em Suppose that $\mathcal{A}_{\lambda}=\begin{bmatrix}
	\mathcal{A}&\lambda\mathcal{I}_m
\end{bmatrix}$ and $\widehat{\mathcal{X}}=\begin{bmatrix}
\mathcal{X}\\ \mathcal{W}
\end{bmatrix}$, where $\mathcal{X}\in\mathbb{R}^{n\times c\times p}$ and $\mathcal{W}\in\mathbb{R}^{m\times c\times p}$. The following theorem implies that the unique solution $\mathcal{X}^*$ of the t-RLS problem (\ref{tensor Tikhonov regularization})
is closely related to the minimum Frobenius norm solution of the following t-LS problem:
\begin{equation}\label{relation theorem tLS}
	\mathrm{min}\left\{\left\|\mathcal{A}_{\lambda}*\widehat{\mathcal{X}}-\mathcal{B}\right\|_F^2:\ \widehat{\mathcal{X}}\in\mathbb{R}^{(m+n)\times c\times p}\right\}.
\end{equation}
\begin{theorem}\label{A_mu}
	Let $\widehat{\mathcal{X}}^*=\begin{bmatrix}
		\mathcal{X}^*\\ \mathcal{W}^*
	\end{bmatrix}$ be the minimum Frobenius norm solution to the t-LS problem (\ref{relation theorem tLS}), then $\mathcal{X}^*$ is the unique solution to the t-RLS problem (\ref{tensor Tikhonov regularization}).
\end{theorem}
\begin{proof}
	From \cite[Theorem 3.2, Theorem 4.5]{Jin2017}, we obtain that
	\begin{equation*}
	\widehat{\mathcal{X}}^*=\mathcal{A}^{\dagger}_{\lambda}*\mathcal{B}=\mathcal{A}_{\lambda}^{\mathrm{T}}*\left(\mathcal{A}_{\lambda}*\mathcal{A}_{\lambda}^{\mathrm{T}}\right)^{\dagger}*\mathcal{B}.
	\end{equation*}
 Applying operator (\ref{operator}) to $\mathcal{A}_{\lambda}$, we know that the matrices
\begin{equation*}
	\bar{\mathcal{A}}_{\lambda}^{(i)}=\begin{bmatrix}
		\bar{\mathcal{A}}^{(i)}&\lambda I_m
	\end{bmatrix}
\end{equation*}
are of full row rank. Then it follows that each $\bar{\mathcal{A}}^{(i)}_{\lambda}(\bar{\mathcal{A}}^{(i)}_{\lambda})^{\mathrm{H}}$ is invertible, which implies that tensor $(\mathcal{A}_{\lambda}*\mathcal{A}_{\lambda}^{\mathrm{T}})$ is invertible, where Lemma \ref{relation} and the definition of tensor inverse are applied.
Therefore, we have
\begin{equation*}
	\widehat{\mathcal{X}}^*=\mathcal{A}_{\lambda}^{\dagger}*\mathcal{B}=\mathcal{A}
_{\lambda}^{\mathrm{T}}*(\mathcal{A}_{\lambda}*\mathcal{A}_{\lambda}^{\mathrm{T}})^{-1}*\mathcal{B}=\begin{bmatrix}
	\mathcal{A}^{\mathrm{T}}\\\lambda\mathcal{I}_m
\end{bmatrix}*(\mathcal{A}*\mathcal{A}^{\mathrm{T}}+\lambda^2\mathcal{I}_n)^{-1}*\mathcal{B}.
\end{equation*}
Then it leads to $\mathcal{X}^*=\mathcal{A}^{\mathrm{T}}*(\mathcal{A}*\mathcal{A}^{\mathrm{T}}+\lambda^2\mathcal{I}_m)^{-1}*\mathcal{B}$, which is the unique solution of (\ref{tensor Tikhonov regularization}) according to (\ref{normal equation c=1}).
\end{proof}
We next consider exploiting the t-QR factorization of $\mathcal{A}_{\lambda}^{\mathrm{T}}$ to represent the solution of t-RLS problem (\ref{tensor Tikhonov regularization}).
\begin{lemma}\label{lemma t-QR　of A_lambda}
	Let
	\begin{equation}\label{t-QR of A_lambda}
		\mathcal{A}_{\lambda}^{\mathrm{T}}=\mathcal{Q}*\mathcal{R}=\begin{bmatrix}
			\mathcal{Q}_1\\\mathcal{Q}_2
		\end{bmatrix}*\mathcal{R}
	\end{equation}
be the t-QR factorization of $\mathcal{A}_{\lambda}^{\mathrm{T}}$, where $\mathcal{Q}\in \mathbb{R}^{(m+n)\times m\times p},\mathcal{Q}_1\in \mathbb{R}^{n\times m\times p},\mathcal{Q}_2\in \mathbb{R}^{m\times m\times p}$ and $\mathcal{R}\in \mathbb{R}^{m\times m\times p}$ is an f-diagonal tensor. Then the unique solution $\mathcal{X}^*$ of the t-RLS problem (\ref{tensor Tikhonov regularization}) can be expressed as
\begin{equation*}
\mathcal{X}^*=\mathcal{Q}_1*\mathcal{R}^{-\mathrm{T}}*\mathcal{B}.
\end{equation*}
\end{lemma}
\begin{proof}
From the proof of Theorem \ref{A_mu}, it is straightforward to show that

\begin{equation*}
	\widehat{\mathcal{X}}^*=\mathcal{A}_{\lambda}^{\mathrm{T}}*(\mathcal{A}_{\lambda}*\mathcal{A}_{\lambda}^{\mathrm{T}})^{-1}*\mathcal{B}=\mathcal{Q}*\mathcal{R}^{-\mathrm{T}}*\mathcal{B}=\begin{bmatrix}
		\mathcal{Q}_1\\\mathcal{Q}_2
	\end{bmatrix}*\mathcal{R}^{-\mathrm{T}}*\mathcal{B},
\end{equation*}
and the proof is complete.
\end{proof}
This formula serves as the theoretical base for the derivation of our incremental algorithm.

\hskip 2em Let $\mathcal{\tilde{A}}_{\lambda}=\begin{bmatrix}
	\mathcal{A}&\lambda\mathcal{I}_m&\mathbf{0}\\\mathcal{A}_1^{\mathrm{T}}&\mathbf{0}&\lambda\mathbf{e}_1
\end{bmatrix}$, $\tilde{\mathcal{B}}=\begin{bmatrix}
\mathcal{B}\\\mathcal{B}_1^{\mathrm{T}}
\end{bmatrix}$. Then we have
	\begin{equation}
	\mathcal{\tilde{A}}_{\lambda}^{\mathrm{T}}=\begin{bmatrix}
	\mathcal{A}^{\mathrm{T}}&\mathcal{A}_1\\\lambda\mathcal{I}_m&\mathbf{0}\\\mathbf{0}&\lambda\mathbf{e_1}
\end{bmatrix}=\begin{bmatrix}
	\mathcal{A}^{\mathrm{T}}_{\lambda}&\mathcal{\widehat{A}}_1\\\mathbf{0}&\lambda\mathbf{e_1}
\end{bmatrix},
	\end{equation}
where $\mathcal{\widehat{A}}_1=\begin{bmatrix}
	\mathcal{A}_1\\\mathbf{0}
\end{bmatrix}\in {\mathbb{R}}^{ (m+n)\times 1\times p}$. 
Let
\begin{equation*}
 \mathcal{H}=\mathcal{Q}^{\mathrm{T}}*\widehat{\mathcal{A}}_1=\mathcal{Q}_1^{\mathrm{T}}*\mathcal{A}_1,\ \mathcal{C}=(\mathcal{I}-\mathcal{Q}*\mathcal{Q}^{\mathrm{T}})*\mathcal{\widehat{A}}_1*\frac{1}{\mu}\mathbf{e}_1=\begin{bmatrix}
 	\mathcal{C}_1\\\mathcal{C}_2
 \end{bmatrix},
\end{equation*}
where $\mathcal{C}_1 \in \mathbb{R}^{n\times 1 \times p},\mathcal{C}_2 \in \mathbb{R}^{m\times 1 \times p}$ and   $\mu=\sqrt{{\left\|(\mathcal{I}-\mathcal{Q}*\mathcal{Q}^{\mathrm{T}})*\mathcal{\widehat{A}}_1\right\|}_F^2+\lambda^2}>0$.

\hskip 2em Therefore, we have
\begin{align*}
	\mathcal{\tilde{A}}^{\mathrm{T}}_{\lambda} &=
	\begin{bmatrix}
		\mathcal{A}^{\mathrm{T}}_{\lambda}&\mathcal{\widehat{A}}_1\\\mathbf{0}&\lambda\mathbf{e}_1
	\end{bmatrix}\\
	&=
	\begin{bmatrix}
		\mathcal{Q}*\mathcal{R}&\mathcal{Q}*\mathcal{H}+(\mathcal{I}-\mathcal{Q}*\mathcal{Q}^\mathrm{T})*\mathcal{\widehat{A}}_1\\\mathbf{0}&\lambda\mathbf{e_1}
	\end{bmatrix}\\
	&=\begin{bmatrix}
		\mathcal{Q}&\mathcal{C}\\\mathbf{0}&\frac{\lambda}{\mu}\mathbf{e_1}
	\end{bmatrix}
	*\begin{bmatrix}
		\mathcal{R}&\mathcal{H}\\\mathbf{0}&\mu\mathbf{e_1}
	\end{bmatrix}\\
	&=\begin{bmatrix}
		\mathcal{\tilde{Q}}_1\\\mathcal{\tilde{Q}}_2
	\end{bmatrix}
	*\mathcal{\tilde{R}},
\end{align*}
where $\mathcal{\tilde{Q}}_1=\begin{bmatrix}
	\mathcal{Q}_1&\mathcal{C}_1
\end{bmatrix}\in {\mathbb{R}}^{ m\times n\times p}$.
Then it follows that
\begin{align}
	\mathcal{\tilde{X}}^* &= \mathcal{\tilde{Q}}_1*\mathcal{\tilde{R}}^{-\mathrm{T}}*\tilde{\mathcal{B}} \notag\\
	&= \begin{bmatrix}	\mathcal{Q}_1&\mathcal{C}_1
	\end{bmatrix}
	*\begin{bmatrix}
		\mathcal{R}&\mathcal{H}\\\mathbf{0}&\mu\mathbf{e_1}
	\end{bmatrix}^{-\mathrm{T}}
	*\begin{bmatrix}
		\mathcal{B}\\\mathcal{B}_1^{\mathrm{T}}
	\end{bmatrix} \notag \\
	&=\begin{bmatrix}
		\mathcal{Q}_1&\mathcal{C}_1
	\end{bmatrix}
	*\begin{bmatrix}
		\mathcal{R}^{-\mathrm{T}}&\mathbf{0}\\(-\frac{1}{\mu})\mathbf{e_1}*\mathcal{H}^\mathrm{T}*\mathcal{R}^{-\mathrm{T}}&\frac{1}{\mu}\mathbf{e_1}
	\end{bmatrix}
	*\begin{bmatrix}
		\mathcal{B}\\\mathcal{B}_1^{\mathrm{T}}
	\end{bmatrix} \notag \\
	&=\mathcal{Q}_1*\mathcal{R}^{-\mathrm{T}}*\mathcal{B}+\mathcal{C}_1*\frac{1}{\mu}\mathbf{e_1}*(\mathcal{B}_1^{\mathrm{T}}-\mathcal{H}^{\mathrm{T}}*\mathcal{R}^{-\mathrm{T}}*\mathcal{B}) \notag \\
	&=\mathcal{X}^*+\mathcal{C}_1*\frac{1}{\mu}\mathbf{e_1}*(\mathcal{B}^{\mathrm{T}}_1-\mathcal{A}^{\mathrm{T}}_1*\mathcal{X}^*)\label{how to avoid QR}.
\end{align}
If $\tilde{\mathcal{X}}^*$ is computed by (\ref{how to avoid QR})
, we need to perform the t-QR factorization to compute $\mathcal{C}_1*\frac{1}{\mu}\mathbf{e_1}$. However, the cost is prohibitively high for the large-scale high-dimensional data. Hence, an alternative approach is investigated. 

\hskip 2em Denote $\mathcal{W}=\mathcal{B}^{\mathrm{T}}_1-\mathcal{A}^{\mathrm{T}}_1*\mathcal{X}^*$ and choose an index $l$ such that $\mathcal{W}_l=\mathcal{W}(:,l,:)$ is invertible. Thus, we have
\begin{equation*}
	\mathcal{\tilde{X}}^*_l=\mathcal{X}^*_l+\mathcal{C}_1*\frac{1}{\mu}\mathbf{e}_1*\mathcal{W}_l,
\end{equation*}
where $\mathcal{X}^*_l=\mathcal{X}^*(:,l,:)$ and $\mathcal{\tilde{X}}^*_l=\mathcal{\tilde{X}}^*(:,l,:)$ such that
\begin{equation}\label{X_t^*}
	\mathcal{\tilde{X}}^*_l=\mathrm{argmin}\left\{{\left\|\mathcal{\tilde{A}*\mathcal{\vec{X}}-\mathcal{\tilde{B}}}(:,l,:)\right\|}^2_F+\lambda^2{\left\|\mathcal{\vec{X}}\right\|}^2_F: \vec{\mathcal{X}}\in \mathbb{R}^{n\times 1\times p}\right\}.
\end{equation}
It follows that
\begin{equation*}
	\mathcal{C}_1*\frac{1}{\mu}\mathbf{e}_1=\left(\mathcal{\tilde{X}}^*_l-\mathcal{X}^*_l\right)*\mathcal{W}_l^{-1}.
\end{equation*}
Substituting it into (\ref{how to avoid QR}) leads to
\begin{equation}
	\mathcal{\tilde{X}}^*=\mathcal{X}^*+(\mathcal{\tilde{X}}^*_l-\mathcal{X}^*_l)*\mathcal{W}_l^{-1}*(\mathcal{B}^{\mathrm{T}}_1-\mathcal{A}_1^{\mathrm{T}}*\mathcal{X}^*).
\end{equation}

\hskip 2em The dominant work in solving (\ref{Tensor regularization new}) to update $\mathcal{X}^*$ is to solve the t-RLS problem with a single lateral slice on the right-hand side. The t-GKT algorithm can be adopted to solve (\ref{X_t^*}). The above yields the t-IRLS framework for solving (\ref{Tensor regularization new}) which is summarized in Algorithm 5.
\begin{algorithm}[htb]
	\caption{t-IRLS algorithm for one new horizontal slice}
	\hspace*{0.02in} {\bf Input:}
	The solution $\mathcal{X^*}$ of (\ref{tensor Tikhonov regularization}), data tensor $\mathcal{A}\in {\mathbb{R}}^{ m\times n\times p}$, response tensor $\mathcal{B}\in {\mathbb{R}}^{ m\times c\times p}$, the new sample $\mathcal{A}_1$, and the corresponding response $\mathcal{B}_1$\\
	\hspace*{0.02in} {\bf Output:}
	The solution $\tilde{\mathcal{X}}^*$ of (\ref{Tensor regularization new})
	\begin{algorithmic}[1]
		\State Compute $\mathcal{W}\gets\mathcal{B}_1^\mathrm{T}-\mathcal{A}_1^\mathrm{T}*\mathcal{X}^*\in\mathbb{R}^{1\times c\times p}$ and index $l$ s.t. $\mathcal{W}(:,l,:)$ is invertible
		\State Apply Algorithm 4 to compute $\mathcal{\tilde{X}}^*_l$ such that
		\begin{equation*}
			\mathcal{\tilde{X}}^*_l=\mathrm{argmin}\left\{{\left\|\mathcal{\tilde{A}*\mathcal{\vec{X}}-\mathcal{\tilde{B}}}(:,l,:)\right\|}^2_F+\lambda^2{\left\|\mathcal{\vec{X}}\right\|}^2_F:\ \vec{\mathcal{X}}\in\mathbb{R}^{n\times 1\times p}\right\}
		\end{equation*}
		\State Compute $\mathcal{\tilde{X}}^*\gets\mathcal{X}^*+(\mathcal{\tilde{X}}^*_l-\mathcal{X}^*_l)*\mathcal{W}_l^{-1}*(\mathcal{B}^\mathrm{T}_1-\mathcal{A}_1^\mathrm{T}*\mathcal{X}^*)$
	\end{algorithmic}
\end{algorithm}
\section{ Numerical results}
\hskip 2em In this section, numerical experiments are performed to investigate the performance of our t-IRLS algorithm. 
All computations are carried out in MATLAB 2020a on a computer with an AMD Ryzen 5 processor and 16 GB RAM running Windows 10.

\hskip 2em Let $\mathcal{X}_{\mathrm{tGKT}}$ and  $\mathcal{X}_{\mathrm{tIRLS}}$ be the computed solutions of (\ref{Tensor regularization new}) by Algorithms 4 and 5 respectively and $\tilde{\mathcal{X}}^*$ refers to the exact solution of ($\ref{Tensor regularization new}$). The relative error
\begin{equation*}
	\mathrm{Err}=\frac{\|\mathcal{X}_{\mathrm{tGKT}}-\tilde{\mathcal{X}}^*\|_F}{\|\tilde{\mathcal{X}}^*\|_F}  ~\mbox{or}~ \frac{\|\mathcal{X}_{\mathrm{tIRLS}}-\tilde{\mathcal{X}}^*\|_F}{\|\tilde{\mathcal{X}}^*\|_F}
\end{equation*}

and the running time in seconds (denoted as “CPU”) are recorded for comparison. Since the purpose of our examples is to illustrate the performance of the t-IRLS algorithm and choosing the regularization parameter $\lambda$ is an important and well-researched topic which is far beyond
the scope of this paper, then we assume that the value of $\lambda$ is already known.
\hskip 2em
$\mathbf{Example\ 1.}$ 
Initially we generate tensor $\mathcal{A}'\in\mathbb{R}^{m\times m\times m}, \mathcal{B}\in {\mathbb{R}}^{ m\times c\times m}, \mathcal{A}_1\in {\mathbb{R}}^{ m\times 1\times m}$ and $\mathcal{B}_1 \in {\mathbb{R}}^{ c\times 1\times m}$ with entries from the standard normal distribution. Let $\mathcal{A}'=\mathcal{U}*\mathcal{S}'*\mathcal{V}^{\mathrm{T}}\in\mathbb{R}^{m\times m\times m}$ be the t-SVD of tensor $\mathcal{A}'$.
The data tensor is constructed by
\begin{equation*}
	\mathcal{A}=\mathcal{U}*\mathcal{S}*\mathcal{V}^{\mathrm{T}},\ \mathrm{with}\  \mathcal{S}(i,i,:)=\begin{cases}
		\mathcal{S}'(i,i,:), & 1\le i\le m-3\\
		(1\mathrm{e}\mbox{-}2)\mathcal{S}'(i,i,:),& \mathrm{otherwise}
	\end{cases}
\end{equation*}
which is of ill-determined tubal rank \cite{ReichelGolubKahanTikhonov2021}, that is, there are many non-vanishing singular tubes of tiny Frobenius norm whose entries are at scale $1\mathrm{e}\mbox{-}2$. Then the t-GKT algorithm and the t-IRLS algorithm are applied for solving the t-RLS problem (\ref{Tensor regularization new}) with $\lambda=1\mathrm{e}2$ respectively.

\hskip 2em The results with various $c$'s are displayed in Tables 1 and 2. Compared with using the t-GKT to directly solve the t-RLS problem (\ref{Tensor regularization new}) while a new sample is inserted, the t-IRLS algorithm can significantly accelerate the solving process while the accuracy of the solution is still maintained. Moreover, with the value of $c$ increasing, we observe that
the advantage of the t-IRLS algorithm becomes more obvious.

\begin{table}[htb]
	\caption{Results for Example 1 with $m=30$ }
	\centering
	\begin{tabular}{ccccc}
		\toprule
		$c$&Method&$\mathrm{Err}$&$k$&CPU                \\
		\toprule
		\multirow{2}*{$10$} &t-IRLS &$3.3533\mathrm{e}\mbox{-}05$ &-&$0.01619$\\
		~&t-GKT&$2.5377\mathrm{e}\mbox{-}04$&$4$&$0.1274$\\
		\toprule
		\specialrule{0em}{0pt}{4pt}
		\multirow{2}*{$100$} &t-IRLS &$5.6209\mathrm{e}\mbox{-}08$&- &$0.03327$\\
		~&t-GKT&$5.9615\mathrm{e}\mbox{-}09$&$7$&$2.979$\\
		\toprule
		\specialrule{0em}{0pt}{4pt}
		\multirow{2}*{$1000$} &t-IRLS &$2.8114\mathrm{e}\mbox{-}13$&- &$0.1077$\\
		~&t-GKT&$1.0558\mathrm{e}\mbox{-}13$&$11$&$62.58$\\
		\toprule
		\specialrule{0em}{0pt}{4pt}
		\multirow{2}*{$10000$} &t-IRLS &$2.4405\mathrm{e}\mbox{-}12$ &-&$0.5624$\\
		~&t-GKT&$1.3011\mathrm{e}\mbox{-}12$&$10$&$763.1$\\
		\toprule
	\end{tabular}
\end{table}

\begin{table}[htb]
	\caption{Results for Example 1 with $m=100$ }
	\centering
	\begin{tabular}{ccccc}
		\toprule
		$c$&Method&$\mathrm{Err}$&$k$&CPU                \\
		\toprule
		\multirow{2}*{$10$} &t-IRLS &$4.2277\mathrm{e}\mbox{-}09$ &-&$0.9550$\\
		~&t-GKT&$1.5866\mathrm{e}\mbox{-}09$&$20$&$9.446$\\
		\toprule
		\specialrule{0em}{0pt}{4pt}
		\multirow{2}*{$50$} &t-IRLS &$2.1504\mathrm{e}\mbox{-}11$&- &$1.669$\\
		~&t-GKT&$2.4890\mathrm{e}\mbox{-}11$&$25$&$70.01$\\
		\toprule
		\specialrule{0em}{0pt}{4pt}
		\multirow{2}*{$100$} &t-IRLS &$3.9632\mathrm{e}\mbox{-}09$&- &$1.002$\\
		~&t-GKT&$1.6088\mathrm{e}\mbox{-}09$&$20$&$95.16$\\
		\toprule
		\specialrule{0em}{0pt}{4pt}
		\multirow{2}*{$500$} &t-IRLS &$5.5120\mathrm{e}\mbox{-}11$ &-&$1.646$\\
		~&t-GKT&$4.6454\mathrm{e}\mbox{-}11$&$25$&$701.9$\\
		\toprule
	\end{tabular}
\end{table}

\hskip 2em
$\mathbf{Example\ 2}.$ \cite{ReichelGolubKahanTikhonov2021}
We use the MATLAB function $\mathtt{baart}$ from Hansen's Regularization Tools \cite{hansen2007regularization}
to generate the matrix $A_1=\mathtt{baart}(m)$ and define $A_2=\mathrm{gallery}('\mathtt{prolate}',m,\alpha)$ with $\alpha=0.46$. Then $A_2$ is a symmetric positive definite ill-conditioned Toeplitz matrix. Let
\begin{equation*}
	\mathcal{A}^{(i)}=A_1(i,1)A_2, ~i=1,2,\ldots,m,
\end{equation*}
then we obtain the tensor $\mathcal{A}\in\mathbb{R}^{m\times m\times m}$. The exact tensor $\mathcal{B}_{\mathrm{true}}$ is generated by $\mathcal{B}_{\mathrm{true}}=\mathcal{A}*\mathcal{X}_{\mathrm{true}}$, where $\mathcal{X}_{\mathrm{true}}\in\mathbb{R}^{m\times c\times m}$ has all entries equal to unity. The noise tensor $\mathcal{E}$ that simulates the error in the data tensor $\mathcal{B}=\mathcal{B}_{\mathrm{true}}+\mathcal{E}$ is given by
\begin{equation*}
\vec{\mathcal{E}}:=\tilde{\delta}\frac{\vec{\mathcal{E}}_{0,j}}{\|\vec{\mathcal{E}}_{0,j}\|_F}\|\vec{\mathcal{B}}_{\mathrm{true},j}\|_F,	
\end{equation*}
where the entries of $\vec{\mathcal{E}}_{0,j}\in\mathbb{R}^{m\times c\times m}$ are normally distributed with zero mean and are scaled to correspond to the specific noise level $\tilde{\delta}=1\mathrm{e}\mbox{-}3$. We next formulate the new sample $\mathcal{A}_1\in\mathbb{R}^{m\times 1\times m}$ and the corresponding data $\mathcal{B}_1\in\mathbb{R}^{c\times1\times m}$ with entries obeying the standard normal distribution. Then we seek the solution of (\ref{Tensor regularization new}) with $\lambda=1/\sqrt{3.91\mathrm{e}\mbox{-}2}$ by the algorithms t-GKT and t-IRLS correspondingly.

\hskip 2em
The computed results are reported in Tables 3 and 4 in the same pattern as in Tables 1 and 2, again revealing a great disparity between algorithms. From the computational cost of the two algorithms under different $c$'s in the table, the t-IRLS algorithm is more suitable for high-dimensional problem as $c$ increases.
Meanwhile, we mention that the solution accuracy of the t-IRLS algorithm largely depends on the accuracy of the algorithm used to solve $\mathcal{X}^*_t$.
\vspace{-0.3cm}
\begin{table}[htb]
	\caption{Results for Example 2 with $m=50$ }
	\centering
	\begin{tabular}{ccccc}
		\toprule
		$c$&Method&$\mathrm{Err}$&$k$&CPU                \\
		\toprule
		\multirow{2}*{$10$} &t-IRLS &$7.9938\mathrm{e}\mbox{-}05$ &-&$0.1092$\\
		~&t-GKT&$7.7431\mathrm{e}\mbox{-}04$&$5$&$1.030$\\
		\toprule
		\specialrule{0em}{0pt}{4pt}
		\multirow{2}*{$50$} &t-IRLS &$8.5286\mathrm{e}\mbox{-}05$&- &$0.09547$\\
		~&t-GKT&$8.3828\mathrm{e}\mbox{-}05$&$5$&$4.401$\\
		\toprule
		\specialrule{0em}{0pt}{4pt}
		\multirow{2}*{$150$} &t-IRLS &$6.4185\mathrm{e}\mbox{-}05$&- &$0.1577$\\
		~&t-GKT&$6.4418\mathrm{e}\mbox{-}05$&$4$&$9.548$\\
		\toprule
		\specialrule{0em}{0pt}{4pt}
		\multirow{2}*{$200$} &t-IRLS &$8.8434\mathrm{e}\mbox{-}04$ &-&$0.1219$\\
		~&t-GKT&$8.8432\mathrm{e}\mbox{-}04$&$4$&$15.31$\\
		\toprule
	\end{tabular}
\end{table}
\vspace{-0.3cm}

\begin{table}[htb]
	\caption{Results for Example 2 with $m=256$ }
	\centering
	\begin{tabular}{ccccc}
		\toprule
		$c$&Method&$\mathrm{Err}$&$k$&CPU                \\
		\toprule
		\multirow{2}*{$10$} &t-IRLS &$1.0123\mathrm{e}\mbox{-}04$ &-&$3.024$\\
		~&t-GKT&$1.0124\mathrm{e}\mbox{-}04$&$4$&$26.96$\\
		\toprule
		\specialrule{0em}{0pt}{4pt}
		\multirow{2}*{$30$} &t-IRLS &$1.0676\mathrm{e}\mbox{-}04$&- &$3.206$\\
		~&t-GKT&$1.0667\mathrm{e}\mbox{-}05$&$4$&$87.05$\\
		\toprule
		\specialrule{0em}{0pt}{4pt}
		\multirow{2}*{$50$} &t-IRLS &$3.9639\mathrm{e}\mbox{-}05$&- &$3.487$\\
		~&t-GKT&$3.9577\mathrm{e}\mbox{-}05$&$4$&$138.3$\\
		\toprule
		\specialrule{0em}{0pt}{4pt}
		\multirow{2}*{$70$} &t-IRLS &$4.0614\mathrm{e}\mbox{-}04$ &-&$2.595$\\
		~&t-GKT&$4.0812\mathrm{e}\mbox{-}04$&$3$&$153.2$\\
		\toprule
	\end{tabular}
\end{table}
\vspace{-0.3cm}
\section{Conclusion}

\hskip 2em The main contribution of this paper is the study of extending the approach of Zhang et al. (2016) to incremental tensor regularized least squares.
The resulting method is computationally efficient than standard methods applied directly to a new tensor system.
Specifically, a t-QR factorization updating technique has been proposed in order to reduce the cost of treating an enlarged tensor.
It must be noted, however, that the dominant work in updating $\mathcal{X}^*$ generally depends on choosing an invertible tube scalar of tensor $\mathcal{W}$,
which is a key issue to obtain $\mathcal{C}_1*\frac{1}{\mu}\mathbf{e}_1$.
Further improvement for lowering the cost of the t-IRLS relies on the classical iterative methods for the t-RLS problem in the second step.
\section*{Acknowledgments}
This work is supported by the National Natural Science Foundation of China under grant No. 11801534.
\bibliographystyle{siam}
\bibliography{tIRLS}
\end{document}